\documentclass[letterpaper,11pt]{article}
\usepackage{color}

\usepackage[small]{titlesec}
\usepackage{theorem,amsmath,amssymb,amscd}
\usepackage[all,cmtip]{xy}
\usepackage{booktabs}
\usepackage[hyperfootnotes=false, colorlinks, linkcolor={blue}, citecolor={magenta}, filecolor={blue}, urlcolor={blue}]{hyperref}
\usepackage[overload]{textcase} 
\usepackage{geometry}
\usepackage{pdflscape}

\theoremstyle{change}
\allowdisplaybreaks
\newcommand{\Z}{{\mathbb Z}}

\newcommand{\C}{{\mathbb C}}

\newcommand{\Fq}{{\mathbb F}_q}
\newcommand{\GL}{{\rm GL}}

\newcommand{\GSp}{{\rm GSp}}

\newcommand{\Sp}{{\rm Sp}}

\newcommand{\Hom}{{\rm Hom}}

\newcommand{\mat}[4]{\begin{bsmallmatrix}#1&#2\\#3&#4\end{bsmallmatrix}}
\newcommand{\qed}{\hspace*{\fill}\rule{1ex}{1ex}}
\newcommand{\forget}[1]{}
\def\qdots{\mathinner{\mkern1mu\raise0pt\vbox{\kern7pt\hbox{.}}\mkern2mu
\raise3.4pt\hbox{.}\mkern2mu\raise7pt\hbox{.}\mkern1mu}}

\newenvironment{proof}{\vspace{1ex}\noindent\emph{Proof.}\hspace{0.5em}}
	{\hfill\qed\vspace{2ex}}
\newenvironment{bsmallmatrix}{\left[\begin{smallmatrix}}{\end{smallmatrix}\right]}

\def\term#1{\textbf{\textit{#1}}}

\makeatletter
\newcommand{\hathat}[1]{%
\begingroup%
  \let\macc@kerna\z@%
  \let\macc@kernb\z@%
  \let\macc@nucleus\@empty%
  \hat{\mathchoice%
    {\raisebox{.2ex}{\vphantom{\ensuremath{\displaystyle #1}}}}%
    {\raisebox{.3ex}{\vphantom{\ensuremath{\textstyle #1}}}}%
    {\raisebox{.16ex}{\vphantom{\ensuremath{\scriptstyle #1}}}}%
    {\raisebox{.14ex}{\vphantom{\ensuremath{\scriptscriptstyle #1}}}}%
    \smash{\hat{#1}}}%
\endgroup%
}
\makeatother

\newtheorem{lemma}{Lemma.}[section]
\newtheorem{theorem}[lemma]{Theorem.}
\newtheorem{corollary}[lemma]{Corollary.}
\newtheorem{proposition}[lemma]{Proposition.}

\begin{document}

\thispagestyle{empty}
\begin{center}
 {\bf\Large Bessel Models for Representations of $\GSp(4, q)$}

 \vspace{3ex}
 Jonathan Cohen


\end{center}

\begin{abstract}
We compute the Bessel models of irreducible representations of the finite group $\GSp(4, q)$. 
\end{abstract}

\section{Introduction}

For $F$ a local or global field, there has been significant work on Bessel models for $\GSp(4,F)$ and their applications to Siegel modular forms and automorphic representations; see \cite{Pi2011}, \cite{Pr2011}, \cite{Ro2016}, and \cite{Sc1995} for some examples. In this article, we consider the analogous situation over finite fields, and our main result is the computation of the Bessel models for all irreducible representations of $\GSp(4, q)$; see Theorem \ref{Main Theorem}. One of our primary motivations is the determination of the (nonsplit) Bessel models for the depth zero supercuspidals of $p$-adic $\GSp(4, F)$, which we will carry out in a forthcoming paper. We also remark that there are some notable differences between the $p$-adic and finite field cases. For example, it is shown in \cite{Ro2016} that all split Bessel models are unique (when they exist) over the $p$-adics, while a consequence of our computations is that the split Bessel models over finite fields can have multiplicity 2; see Corollary \ref{corollary} below. On the other hand, over both $p$-adic and finite fields, the generic representations admit every split Bessel model and (if $q>3$ in the finite field case) this characterizes genericity. 

We now briefly outline the contents of this article. In section 2 we introduce the requisite notation, in section 3 we carry out a technical preliminary computation, and in section 4 we prove the main result. 

\section{Notations}

Let $\zeta_n$ denote a complex primitive $n$th root of unity and $\delta_{i,j}$ denote the Kronecker delta for $i,j$ varying over some set. Let $q$ be a power of a prime $p$ and $\Fq$ denote the field of order $q$. If $q$ is even then the group homomorphism $\Fq\to\Fq$ given by $x\mapsto x^2+x$ is two-to-one. Let $\Fq^\circ$ be its image and $\varepsilon:\Fq\to\{\pm1\}$ be the homomorphism with kernel~$\Fq^\circ$. Let $\xi$ be a fixed element of $\Fq$ that is not in $\Fq^\circ$; so $\varepsilon(t^2+t)=1$ and $\varepsilon(t^2+t+\xi)=-1$ for all~$t\in\Fq$. If $q$ is odd let $\xi$ be a fixed nonsquare in $\Fq^\times$. 

 Fix the symplectic form
\begin{equation}\label{Jdefeq}
 J=\begin{bsmallmatrix}&&&1\\&&1\\&-1\\-1\end{bsmallmatrix}
\end{equation}
and define \begin{equation}\label{GSp4defeq}
 G:=\GSp(4,q)=\{g\in\GL(4,q)\mid\ ^tgJg=\mu(g)J\text{ for some }\mu(g)\in \Fq^\times\}.
\end{equation}
The kernel of the multiplier homomorphism $\mu:\GSp(4,q)\to \Fq^\times$ is the symplectic group $\Sp(4,q)$. The center $Z$ of $\GSp(4,q)$ consists of scalar invertible matrices. If $q$ is even then $G = \Sp(4, q)\times Z$. 
We further define the Siegel parabolic subgroup
\begin{equation}\label{Pdefeq}
 P=G\cap\begin{bsmallmatrix} *&*&*&*\\ *&*&*&*\\&&*&*\\&&*&*\end{bsmallmatrix}
\end{equation}
and its unipotent radical $N$, given by
\begin{equation}\label{Ndefeq}
 N=\{\begin{bsmallmatrix}1&&y&z\\&1&x&y\\&&1\\&&&1\end{bsmallmatrix}\mid x,y,z\in \Fq\}.
\end{equation} If $A =\begin{bsmallmatrix} a & b  \\
c & d 
\end{bsmallmatrix} \in \GL(2, q)$ then let $A' := \frac{1}{ad-bc} \begin{bsmallmatrix}
a & -b \\ -c  & d
\end{bsmallmatrix} $ so $\begin{bsmallmatrix}
A & \\ & u A'
\end{bsmallmatrix}\in P$ with multiplier $u\in \Fq^\times$. 

Let $\psi$ be a non-trivial character of $\Fq$. For $a,b,c\in\Fq$, let $\psi_{a,b,c}$ be the character of $N$ given by
\begin{equation}\label{psiabceq}
 \psi_{a,b,c}(\begin{bsmallmatrix}1&&y&z\\&1&x&y\\&&1\\&&&1\end{bsmallmatrix})=\psi(ax+by+cz)
\end{equation}
for $x,y,z\in\Fq$. We denote by $\C_{a,b,c}$ the one-dimensional representation of $N$ given by the character~$\psi_{a,b,c}$. We also define the group 
\begin{equation} \label{Tdefeq}
T = \{\begin{bsmallmatrix} r-bs & -as \\ cs  & r\\  && r -bs&  as \\ &  & -cs  &  r\end{bsmallmatrix}\mid r,s\in\Fq , \ r^2-brs+s^2ac\in \Fq^\times \}. 
\end{equation} 
The reader may readily check that $T$ normalizes $N$ and preserves the form $ax+by+cz$. Let $\chi:T \to \C^\times$ be a character. 
The map $tn\mapsto\chi(t)\psi_{a,b,c}(n)$ defines a character of the \term{Bessel subgroup} 
\begin{equation}
R:=TN,
\end{equation} which we denote by~$\chi\otimes\psi_{a,b,c}$. Let $\C_{\chi,a,b,c}$ be the one-dimensional representation of $R$ given by the character $\chi\otimes\psi_{a,b,c}$. 

The irreducible characters of $G$ have been determined in \cite{Shinoda1982} (if $q$ is odd) and \cite{Enomoto1972} (if $q$ is even). Some of the mistakes in Table~IV-2 of~\cite{Enomoto1972} have been corrected in Table~7 of~\cite{dabbaghianabdoly2007}; we remark that the values for $\theta_1$ and $\theta_2$ there are transposed.  

%

\section{Fourier coefficients}
The main result of this paper is a computation of $\dim\Hom_R(V_\sigma, \C_{\chi, a,b,c})$ for all irreducible representations $(\sigma, V_\sigma)$ of $G$, when $b^2-4ac\neq 0$. To carry this out it is convenient to first prove a preliminary result. 
\begin{proposition}\label{theta2psiabcprop}
 The dimensions of the spaces $\Hom_N(V_\sigma,\C_{a,b,c})$ for all irreducible representations $(\sigma,V_\sigma)$ of $\GSp(4,q)$ are as given in Table~\ref{GSp4qoddtable} (if $q$ is odd) and Table~\ref{GSp4qeventable} (if $q$ is even).
\end{proposition}
\begin{proof} 
By character theory we have 
\begin{align*}
 \dim\Hom_N(V_\sigma,\C_{a,b,c}) 
 =\frac1{q^3}\sum_{x,y,z\in\Fq}\psi(-ax-by-cz)\theta(\begin{bsmallmatrix}1&&y&z\\&1&x&y\\&&1\\&&&1\end{bsmallmatrix}),
\end{align*}
where $\theta$ is the trace character of~$\sigma$. For $2\times2$ matrices $X,Y$ over $\Fq$, write $X\sim Y$ if there exist $A\in\GL(2,\Fq)$ and $u\in\Fq^\times$ with $uAX(A')^{-1}=Y$.

Suppose that $q$ is odd. Then, for $X=\mat{y}{z}{x}{y}$,
\begin{align*}
 &\text{if $X\neq0$ and $\det(X)=0$, then $X\sim\mat{0}{1}{0}{0}$,}\\
 &\text{if $\det(X)\in\Fq^{\times2}$, then $X\sim\mat{1}{0}{0}{1}$,}\\
 &\text{if $\det(X)\in\xi\Fq^{\times2}$, then $X\sim\mat{0}{-\xi}{1}{0}$.}
\end{align*}
Here, $\xi$ is a fixed non-square in $\Fq^\times$. Hence
\begin{align*}
 & \dim\Hom_N(V_\sigma,\C_{a,b,c})=\frac1{q^3}\Bigg(\theta(\begin{bsmallmatrix}1\\&1\\&&1\\&&&1\end{bsmallmatrix})+\sum_{\substack{x,y,z\in\Fq\\y^2-xz=0\\(x,y,z)\neq(0,0,0)}}\psi(-ax-by-cz)\cdot\theta(\begin{bsmallmatrix}1&&&1\\&1\\&&1\\&&&1\end{bsmallmatrix})\\
 &\quad+\sum_{\substack{x,y,z\in\Fq\\y^2-xz\in\Fq^{\times2}}}\psi(-ax-by-cz)\cdot\theta(\begin{bsmallmatrix}1&&1\\&1&&1\\&&1\\&&&1\end{bsmallmatrix})+\sum_{\substack{x,y,z\in\Fq\\y^2-xz\in\xi\Fq^{\times2}}}\psi(-ax-by-cz)\cdot\theta(\begin{bsmallmatrix}1&&&-\xi\\&1&1\\&&1\\&&&1\end{bsmallmatrix})\Bigg).
\end{align*}
The four matrices in this equation represent the conjugacy classes denoted by $A_0$, $A_1$, $A_{21}$ and $A_{22}$ in~\cite{Shinoda1982}. Hence the assertion follows from Table~5-11 of~\cite{Shinoda1982} and Lemma~\ref{FCoddlemma} below.

Now suppose that $q$ is even. Then, for $X=\mat{y}{z}{x}{y}$,
\begin{align*}
 &\text{if $X\neq0$ and $\det(X)=0$, then $X\sim\mat{0}{1}{0}{0}$,}\\
 &\text{if $\det(X)\in\Fq^\times$ and $x=z=0$, then $X\sim\mat{1}{0}{0}{1}$,}\\
 &\text{if $\det(X)\in\Fq^\times$ and $(x,z)\neq(0,0)$, then $X\sim\mat{1}{1}{0}{1}$.}
\end{align*}
Hence
\begin{align*}
 & \dim\Hom_N(V_\sigma,\C_{a,b,c})=\frac1{q^3}\Bigg(\theta(\begin{bsmallmatrix}1\\&1\\&&1\\&&&1\end{bsmallmatrix})+\sum_{\substack{x,y,z\in\Fq\\y^2-xz=0\\(x,y,z)\neq(0,0,0)}}\psi(-ax-by-cz)\cdot\theta(\begin{bsmallmatrix}1&&&1\\&1\\&&1\\&&&1\end{bsmallmatrix})\\
 &\quad+\sum_{\substack{x,y,z\in\Fq\\y^2-xz\neq0\\(x,z)=(0,0)}}\psi(-ax-by-cz)\cdot\theta(\begin{bsmallmatrix}1&&1\\&1&&1\\&&1\\&&&1\end{bsmallmatrix})+\sum_{\substack{x,y,z\in\Fq\\y^2-xz\neq0\\(x,z)\neq(0,0)}}\psi(-ax-by-cz)\cdot\theta(\begin{bsmallmatrix}1&&1&1\\&1&&1\\&&1\\&&&1\end{bsmallmatrix})\Bigg).
\end{align*}
The four matrices in this equation represent the conjugacy classes denoted by $A_1$, $A_2$, $A_{31}$ and $A_{32}$ in~\cite{Enomoto1972}. Hence the assertion follows from Table~IV-2 of~\cite{Enomoto1972} and Lemma~\ref{FCevenlemma} below.\end{proof}

\begin{lemma}\label{FCoddlemma}
 Let $q$ be odd and $a,b,c\in\Fq$. Then
 \begin{enumerate}
  \item 
    \begin{equation}\label{FCoddlemmaeq1}
     \sum_{\substack{x,y,z\in\Fq\\y^2-xz=0\\(x,y,z)\neq(0,0,0)}}\psi(-ax-by-cz)
     =\begin{cases}
       q^2-1&\text{if }a=b=c=0,\\
       -1&\text{if }b^2-4ac=0,\:(a,c)\neq(0,0),\\
       q-1&\text{if }b^2-4ac\in\Fq^{\times2},\\
       -q-1&\text{if }b^2-4ac\in\xi\Fq^{\times2}.
      \end{cases}
    \end{equation}
  \item 
    \begin{equation}\label{FCoddlemmaeq2}
     \sum_{\substack{x,y,z\in\Fq\\ y^2-xz\in\Fq^{\times2}}}\psi(-ax-by-cz)
     =\begin{cases}
       \displaystyle\frac{q(q^2-1)}2&\text{if }a=b=c=0,\\[2ex]
       \displaystyle\frac{q(q-1)}2&\text{if }b^2-4ac=0,\:(a,c)\neq(0,0),\\[1ex]
       -q&\text{if }b^2-4ac\in\Fq^{\times2},\\[1ex]
       0&\text{if }b^2-4ac\in\xi\Fq^{\times2}.
      \end{cases}
    \end{equation}
  \item 
    \begin{equation}\label{FCoddlemmaeq3}
     \sum_{\substack{x,y,z\in\Fq\\y^2-xz\in\xi\Fq^{\times2}}}\psi(-ax-by-cz)
     =\begin{cases}
       \displaystyle\frac{q(q-1)^2}2&\text{if }a=b=c=0,\\[2ex]
       \displaystyle-\frac{q(q-1)}2&\text{if }b^2-4ac=0,\:(a,c)\neq(0,0),\\[1ex]
       0&\text{if }b^2-4ac\in\Fq^{\times2},\\[1ex]
       q&\text{if }b^2-4ac\in\xi\Fq^{\times2}.
      \end{cases}
    \end{equation}
 \end{enumerate}
\end{lemma}
\begin{proof}
Since
\begin{equation*}
 \sum_{x,y,z\in\Fq}\psi(-ax-by-cz)=
 \begin{cases}
  q^3&\text{if }a=b=c=0,\\
  0&\text{otherwise},
 \end{cases}
\end{equation*}
iii) will follow once we prove i) and ii).

i) We calculate
\begin{align*}
 \sum_{\substack{x,y,z\in\Fq\\y^2-xz=0\\(x,y,z)\neq(0,0,0)}}\psi(-ax-by-cz) 
 &=\sum_{\substack{x\in\Fq^\times\\y\in\Fq}}\psi(-x(a+by+cy^2))+\sum_{z\in\Fq^\times}\psi(-cz).
\end{align*}
If $c=0$, then we see easily that
\begin{equation*}
 \sum_{\substack{x,y,z\in\Fq\\y^2-xz=0\\(x,y,z)\neq(0,0,0)}}\psi(-ax-by-cz)
 =\begin{cases}
   q-1&\text{if }b\neq0,\\
   -1&\text{if }a\neq0,\:b=0,\\
   q^2-1&\text{if }a=b=0.
  \end{cases}
\end{equation*}
Suppose that $c\neq0$. Then
\begin{align*}
 \sum_{\substack{x,y,z\in\Fq\\y^2-xz=0\\(x,y,z)\neq(0,0,0)}}\psi(-ax-by-cz)&=\sum_{\substack{x\in\Fq^\times\\y\in\Fq}}\psi(-x(a+by+cy^2))-1\\
 &=\begin{cases}
    1\cdot(q-1)+(q-1)(-1)-1&\text{if }b^2-4ac=0,\\
    2\cdot(q-1)+(q-2)(-1)-1&\text{if }b^2-4ac\in\Fq^{\times2},\\
    0\cdot(q-1)+q(-1)-1&\text{if }b^2-4ac\in\xi\Fq^{\times2},\\
   \end{cases}\\
 &=\begin{cases}
    -1&\text{if }b^2-4ac=0,\\
    q-1&\text{if }b^2-4ac\in\Fq^{\times2},\\
    -q-1&\text{if }b^2-4ac\in\xi\Fq^{\times2}.
   \end{cases}
\end{align*}
This concludes the proof of i).

ii) We calculate
\begin{align*}
 \sum_{\substack{x,y,z\in\Fq\\y^2-xz\in\Fq^{\times2}}}\psi(-ax-by-cz)
 &=\frac12\sum_{u\in\Fq^\times}\sum_{\substack{x,y,z\in\Fq\\y^2-xz=1}}\psi(u(ax+by+cz)). 
 \end{align*} The $x=0$ terms contribute \begin{align*}
\frac{1}{2}\sum\limits_{u\in \Fq^\times} \sum\limits_{z\in \Fq}\psi(cz)(\psi(ub)+\psi(-ub)) &= q\delta_{c,0}(-1+q\delta_{b,0}).
 \end{align*} The $x\neq 0$ terms contribute \begin{align*}
 \frac{1}{2}\sum\limits_{u,x\in \Fq^\times}\sum\limits_{y\in \Fq} \psi(u(ax+by +c (\frac{y^2-1}{x}) ))
 &= \frac{1}{2}\sum\limits_{u,x\in \Fq^\times}\sum\limits_{y\in \Fq} \psi(u(ax^2+byx +cy^2-c )).
 \end{align*} 
 
First suppose that $c=0$. Then\begin{align*}
\frac{1}{2}\sum\limits_{u,x\in \Fq^\times}\sum\limits_{y\in \Fq} \psi(u(ax^2+byx +cy^2-c ))
&= \frac{1}{2} \sum\limits_{u,x\in \Fq^\times}\sum\limits_{y\in \Fq} \psi(ua+by) &= \frac{q(q-1)}{2}\delta_{b,0}(-1+q\delta_{a,0}).
 \end{align*} 

This confirms the formulas when $c=0$. From now on suppose that $c\neq0$. Then
\begin{align}\label{FCoddlemmaeq4}
\frac{1}{2}\sum\limits_{u,x\in \Fq^\times}\sum\limits_{y\in \Fq} \psi(u(ax^2+byx +cy^2-c )) &= \frac{1}{2}\sum\limits_{u,x\in \Fq^\times}\sum\limits_{y\in \Fq} \psi(u(y^2-x^2(b^2-4ac)-1 )).
 \end{align} 
Suppose that $b^2-4ac=0$. Then
\begin{align*}
 \text{\eqref{FCoddlemmaeq4}}&=\frac{(q-1)}{2}\sum_{\substack{u\in\Fq^\times\\y\in\Fq}}\psi(u(y^2-1))
 = \frac{(q-1)}{2}( 2(q-1)-(q-2) ) =\frac{q(q-1)}{2}.
\end{align*}
Next suppose that $b^2-4ac\in\xi\Fq^{\times2}$. Then using the fact that there are $q+1$ solutions $(x,y)\in \Fq\times \Fq$ to the equation $y^2-x^2\xi = 1$, we obtain 
\begin{align*}
  2*\text{\eqref{FCoddlemmaeq4}}&=\sum_{\substack{u\in\Fq^\times\\x,y\in\Fq}}\psi\Big(-u\Big(y^2-x^2\xi-1\Big)\Big)-\sum_{\substack{u\in\Fq^\times\\y\in\Fq}}\psi(-u(y^2-1))\\
  & (q-1)(q+1) - (q^2-(q+1)) - (2(q-1)-(q-2))=0
\end{align*}
Finally, suppose that $b^2-4ac\in\Fq^{\times2}$. Then
\begin{align*}
 \text{\eqref{FCoddlemmaeq4}}&=\frac{1}{2}\sum_{\substack{u,x\in\Fq^\times\\y\in\Fq}}\psi\Big(u\big(y^2-x^2-1\big)\Big)
 =\frac{1}{2}\sum_{\substack{u,x\in\Fq^\times\\y\in\Fq}}\psi\Big(u\big(y^2+2yx-1\big)\Big)\\
 &=\frac{1}{2}\sum_{u,x,y\in\Fq^\times}\psi\Big(u\big(y^2+x-1\big)\Big)+\frac{1}{2}\sum_{u,x\in\Fq^\times}\psi(-u)\\
 &= \sum\limits_{u,x\in \Fq^\times} \psi(ux) + \frac{(q-3)}{2}\sum\limits_{u,x\in \Fq^\times} \psi(u(x+1)) -\frac{q-1}{2}\\
 &= -(q-1)+ \frac{q-3}{2}( q-1-(q-2) )- \frac{q-1}{2} = -q.
\end{align*}
This concludes the proof.
\end{proof}

\begin{lemma}\label{FCevenlemma}
 Let $q$ be even and $a,b,c\in\Fq$. Then
 \begin{enumerate}
  \item 
    \begin{equation}\label{FCevenlemmaeq1}
     \sum_{\substack{x,y,z\in\Fq\\y^2-xz=0\\(x,y,z)\neq(0,0,0)}}\psi(-ax-by-cz)
     =\begin{cases}
       q^2-1&\text{if }a=b=c=0,\\
       -1&\text{if }b=0,\:(a,c)\neq(0,0),\\
       q-1&\text{if }b\neq0,\:\varepsilon(acb^{-2})=1,\\
       -q-1&\text{if }b\neq0,\:\varepsilon(acb^{-2})=-1.
      \end{cases}
    \end{equation}
  \item 
    \begin{equation}\label{FCevenlemmaeq2}
     \sum_{\substack{x,y,z\in\Fq\\y^2-xz\neq0\\(x,z)=(0,0)}}\psi(-ax-by-cz)
     =\begin{cases}
       q-1&\text{if }a=b=c=0,\\
       q-1&\text{if }b=0,\:(a,c)\neq(0,0),\\
       -1&\text{if }b\neq0,\:\varepsilon(acb^{-2})=1,\\
       -1&\text{if }b\neq0,\:\varepsilon(acb^{-2})=-1.
      \end{cases}
    \end{equation}
  \item 
    \begin{equation}\label{FCevenlemmaeq3}
     \sum_{\substack{x,y,z\in\Fq\\y^2-xz\neq0\\(x,z)\neq(0,0)}}\psi(-ax-by-cz)
     =\begin{cases}
       (q-1)(q^2-1)&\text{if }a=b=c=0,\\
       -q+1&\text{if }b=0,\:(a,c)\neq(0,0),\\
       -q+1&\text{if }b\neq0,\:\varepsilon(acb^{-2})=1,\\
       q+1&\text{if }b\neq0,\:\varepsilon(acb^{-2})=-1.
      \end{cases}
    \end{equation}
 \end{enumerate}
\end{lemma}
\begin{proof}
Since
\begin{equation*}
 \sum_{x,y,z\in\Fq}\psi(-ax-by-cz)=
 \begin{cases}
  q^3&\text{if }a=b=c=0,\\
  0&\text{otherwise},
 \end{cases}
\end{equation*}
iii) will follow once we prove i) and ii). Note that ii) is an easy exercise. To prove~i), we calculate
\begin{align*}
 \sum_{\substack{x,y,z\in\Fq\\y^2-xz=0\\(x,y,z)\neq(0,0,0)}}\psi(-ax-by-cz)&=\sum_{\substack{x\in\Fq^\times\\y\in\Fq}}\psi\Big(-ax-by-c\frac{y^2}x\Big)+\sum_{z\in\Fq^\times}\psi(-cz)\\
 &=\sum_{\substack{x\in\Fq^\times\\y\in\Fq}}\psi(-x(a+by+cy^2))+\sum_{z\in\Fq^\times}\psi(-cz).
\end{align*}
If $c=0$, then we see easily that
\begin{equation*}
 \sum_{\substack{x,y,z\in\Fq\\y^2-xz=0\\(x,y,z)\neq(0,0,0)}}\psi(-ax-by-cz)
 =\begin{cases}
   q-1&\text{if }b\neq0,\\
   -1&\text{if }a\neq0,\:b=0,\\
   q^2-1&\text{if }a=b=0.
  \end{cases}
\end{equation*}
Suppose that $c\neq0$. Then
\begin{equation*}
 \sum_{\substack{x,y,z\in\Fq\\y^2-xz=0\\(x,y,z)\neq(0,0,0)}}\psi(-ax-by-cz)=\sum_{\substack{x\in\Fq^\times\\y\in\Fq}}\psi(-x(a+by+cy^2))-1.
\end{equation*}
If $a=b=0$, this is easily seen to be $-1$. Suppose that $b=0$ and $a\neq0$. Then $a+by+cy^2=0$ for exactly one value of $y$, so that
\begin{equation*}
 \sum_{\substack{x,y,z\in\Fq\\y^2-xz=0\\(x,y,z)\neq(0,0,0)}}\psi(-ax-by-cz)=1(q-1)+(q-1)(-1)-1=-1.
\end{equation*}
Suppose that $b\neq0$ (and still $c\neq0$). Then
\begin{equation*}
 a+by+cy^2=0\;\Longleftrightarrow\;bcy+(cy)^2=ca\;\Longleftrightarrow\;cb^{-1}y+(cb^{-1}y)^2=cab^{-2}
\end{equation*}
If $\varepsilon(cab^{-2})=1$, then this equation has exactly two solutions, otherwise none. Hence
\begin{align*}
 \sum_{\substack{x,y,z\in\Fq\\y^2-xz=0\\(x,y,z)\neq(0,0,0)}}\psi(-ax-by-cz)&=\begin{cases}
    2\cdot(q-1)+(q-2)(-1)-1&\text{if }\varepsilon(cab^{-2})=1,\\
    0\cdot(q-1)+q(-1)-1&\text{if }\varepsilon(cab^{-2})=-1,\\
   \end{cases}\\
 &=\begin{cases}
    q-1&\text{if }\varepsilon(cab^{-2})=1,\\
    -q-1&\text{if }\varepsilon(cab^{-2})=-1.
   \end{cases}
\end{align*}
This concludes the proof of i).
\end{proof}
\section{Bessel models}

We now state and prove the main result of this paper. \begin{theorem}\label{Main Theorem}
If $b^2-4ac\neq 0$ then the dimensions of the spaces $\Hom_R(V_\sigma, \C_{\chi,a,b,c})$ for all irreducible representations $(\sigma, V_\sigma)$ of $\GSp(4, q)$ are as given in Table \ref{Bessel odd} if $q$ is odd and in Table \ref{Bessel even} if $q$ is even.  
\end{theorem}

\begin{proof} 
Let $\omega$ be the central character and $\theta$ the trace character of $\sigma$. 
By character theory, \begin{eqnarray}
 \dim\Hom_R(V_\sigma,\C_{\chi,a,b,c}) 
 &=&\frac1{|T|q^3}\sum_{\substack{t\in T \\ x,y,z\in \Fq }}\chi(t)^{-1}\psi(-ax-by-cz)\theta(t\begin{bsmallmatrix}1&&y&z\\&1&x&y\\&&1\\&&&1\end{bsmallmatrix})\nonumber \\
 &=& S_1 + S_2\nonumber 
\end{eqnarray} where \begin{align*}
 S_1&=\frac1{|T|q^3}\sum_{t\in Z}\sum_{x,y,z\in\Fq}\chi(t)^{-1}\psi(-ax-by-cz)\theta(t\begin{bsmallmatrix}1&&y&z\\&1&x&y\\&&1\\&&&1\end{bsmallmatrix}),\\
 S_2&=\frac1{|T|q^3}\sum_{t\in T\setminus Z}\sum_{x,y,z\in\Fq}\chi(t)^{-1}\psi(-ax-by-cz)\theta(t\begin{bsmallmatrix}1&&y&z\\&1&x&y\\&&1\\&&&1\end{bsmallmatrix}).\end{align*} Clearly if $\chi|_Z\neq \omega$ then $S_1=S_2=0$, so assume that $\chi|_Z=\omega$. Then \begin{align*}
 S_1 = \frac{q-1}{|T|q^3}\sum\limits_{x,y,z\in \Fq}\psi(-ax-by-cz)\theta(\begin{bsmallmatrix}1&&y&z\\&1&x&y\\&&1\\&&&1\end{bsmallmatrix}).
 \end{align*} Up to the factor of $(q-1)/|T|$, this was computed in the previous section.

Suppose first that $q$ is odd. Evidently $T\cong \Fq^\times \times \Fq^\times$ if $b^2-4ac$ is a square and $T\cong \mathbb{F}_{q^2}^\times$ otherwise. Let  \begin{equation}
 g=\begin{bsmallmatrix}r-sb/2 & -sa \\ sc & r+sb/2 \\&& r-sb/2 & sa \\&& -sc & r+sb/2\end{bsmallmatrix}\begin{bsmallmatrix}1&&y&z\\&1&x&y\\&&1\\&&&1\end{bsmallmatrix}\in TN.
\end{equation} The eigenvalues of $g$ are $\alpha_\pm =\alpha_\pm(g)= r\pm s(b^2/4-ac)^{1/2}$. If $s=0$, the conjugacy class of $g$ was determined in the previous section. Suppose $ s\neq 0$. If $b^2-4ac$ is a square then \begin{equation}
 g\text{ is in a conjugacy class of type }
 \begin{cases}
  D_0(\alpha_+,\alpha_-) &\text{if } ax+by+cz=0,\\
 D_1(\alpha_+,\alpha_-) &\text{if } ax+by+cz\neq 0,\\
 \end{cases}
\end{equation} while if $b^2-4ac$ is a nonsquare then \begin{equation}
 g\text{ is in a conjugacy class of type }
 \begin{cases}
  F_0(\alpha_+) &\text{if } ax+by+cz=0,\\
 F_1(\alpha_+) &\text{if } ax+by+cz\neq 0,\\
 \end{cases}.
\end{equation} Here we are using the notation of \cite{Shinoda1982}. 

Suppose $b^2-4ac$ is a square, so $|T|=(q-1)^2$. Define the subgroups $T_\pm =\{ t\in T| \ \alpha_\mp(t) =1  \}$ and write $\chi_\pm := \chi|_{T_\pm}$. Clearly $T=T_+T_-=T_\pm Z\cong T_+ \times T_-\cong Z\times T_\pm$ and $T_\pm \cong \Fq^\times$ via $t\mapsto \alpha_\pm(t)$. The reader can check that if $t\in T_\pm$ then $\alpha_\pm(t) = 2r-1$. Let $t_d$ denote the unique element of $T_+$ with eigenvalues $d$ and $1$.  
From the above we have \begin{align*}
S_2&=\frac1{q^3(q-1)^2}\sum\limits_{t\in T\setminus Z}\Bigg(\sum\limits_{\substack{x,y,z\in\Fq\\ax+by+zc=0}}\chi(t)^{-1}\psi(-ax-by-cz)\theta(D_0(\alpha_+(t), \alpha_-(t) ) )\nonumber \\
 &\hspace{25ex}+\sum\limits_{\substack{x,y,z\in\Fq\\ax+by+zc\neq0}}\chi(t)^{-1}\psi(-ax-by-cz)\theta(D_1(\alpha_+(t), \alpha_-(t) ) )\Bigg)\nonumber \\
&=\frac{1}{q(q-1)^2}\sum\limits_{t\in T\setminus Z} \chi(t)^{-1} \left[\theta (D_0(\alpha_+(t),\alpha_-(t) )) - \theta (D_1(\alpha_+(t),\alpha_-(t)) )  \right]\nonumber\\
&= \frac{1}{q(q-1)^2} \sum\limits_{ t\in T\setminus Z} \chi\left( \frac{1}{\alpha_-(t)}t\right)^{-1} \left[\theta (D_0(\alpha_+(t)/\alpha_-(t),1 )) - \theta (D_1(\alpha_+(t)/\alpha_-(t) , 1) )  \right]\nonumber\\
&= \frac{1}{q(q-1)}\sum\limits_{1\neq t\in T_+} \chi(t)^{-1} \left[\theta (D_0(\alpha_+(t), 1)) - \theta (D_1(\alpha_+(t),1 ))  \right]\nonumber\\
&= \frac{1}{q(q-1)}\sum\limits_{ r\in \Fq\setminus \{1/2, 1\}} \chi(t_{2r-1})^{-1} \left[\theta (D_0(2r-1, 1)) - \theta (D_1(2r-1,1 ))  \right]\nonumber\\
&= \frac{1}{q(q-1)}\sum\limits_{1\neq d\in \Fq^\times} \chi_+(d)^{-1} \left[\theta (D_0(d, 1)) - \theta (D_1(d,1 ))  \right].\nonumber
\end{align*} 
In the last line we identify $d$ and $t_d$ 
via the isomorphism $T_+\cong \Fq^\times$. 
The value of $S_2$ can now readily be computed from the character table in \cite{Shinoda1982}. 

Now suppose $b^2-4ac$ is not a square, so $|T|=q^2-1$. 
For $t\in T\setminus Z$, let $\alpha_t, \alpha_t^q$ denote the eigenvalues of $t$. By the above considerations,
\begin{align*}
 S_2&=\frac1{q^3(q^2-1)}\sum_{t\in T\setminus Z}\Bigg(\sum_{\substack{x,y,z\in\Fq\\ax+by+zc=0}}\chi(t)^{-1}\psi(-ax-by-cz)\theta(F_0(\alpha_t) )\\
 &\hspace{25ex}+\sum_{\substack{x,y,z\in\Fq\\ax+by+zc\neq0}}\chi(t)^{-1}\psi(-ax-by-cz)\theta(F_1(\alpha_t) )\Bigg)\\
 &=\frac1{q^3(q^2-1)}\sum_{t\in T\setminus Z}\chi(t)^{-1}\Bigg(q^2\theta(F_0(\alpha_t) )\\
 &\hspace{5ex}+\Bigg(\sum_{x,y,z\in\Fq}\psi(-ax-by-cz)-\sum_{\substack{x,y,z\in\Fq\\ax+by+zc=0}}\psi(-ax-by-cz)\Bigg)\theta(F_1(\alpha_t) )\Bigg)\\
 &=\frac1{q(q^2-1)}\sum_{t\in T\setminus Z}\chi(t)^{-1}\big(\theta(F_0(\alpha_t) )-\theta(F_1(\alpha_t) )\big). 
\end{align*} The value of $S_2$ can now readily be computed from the character tables in \cite{Shinoda1982}. 


Now suppose $q$ is even, so $b\neq 0$. Following \cite{Enomoto1972} we fix elements $\gamma$ and $\eta$ of $\mathbb{F}_{q^2}$ of orders $q-1$ and $q+1$, respectively. Evidently $T\cong \Fq^\times\times \Fq^\times$ if $\varepsilon(acb^{-2})=1$ and $T\cong \mathbb{F}_{q^2}^\times$ otherwise.   
Let \begin{equation}
g = \begin{bsmallmatrix} r+bs & as \\ sc  & r\\&& r+bs & as \\&&  cs  & r s\end{bsmallmatrix}\begin{bsmallmatrix}1&&y&z\\&1&x&y\\&&1\\&&&1\end{bsmallmatrix}\in TN.
\end{equation} Note $\mu(g)^{-1/2}g \in \Sp(4,q)$ and if $s=0$ then the conjugacy class of $\mu(g)^{-1/2}g = r^{-1/2}g$ was found in the previous section. So assume $s\neq0$. If $\varepsilon(acb^{-2})=1$ then the eigenvalues of $g$ are $\gamma^{\pm i}$ for some $i$ and 
\begin{equation}
\mu(g)^{-1/2}g\text{ is in a conjugacy class of type }
 \begin{cases}
  C_2(i) &\text{if } ax+by+cz=0,\\
D_2(i) &\text{if } ax+by+cz\neq 0,\\
 \end{cases}\end{equation} while if $\varepsilon(acb^{-2})=-1$ then the eigenvalues of $g$ are $\eta^{\pm i}$ for some $i$ and \begin{equation}\label{finbeseq39}
   \mu(g)^{-1/2} g\text{ is in a conjugacy class of type }
    \begin{cases}
     C_4(i)&\text{if }ax+by+cz=0,\\
     D_4(i)&\text{if }a x+by+cz\neq0.
    \end{cases}
   \end{equation} Here we are using the notation of \cite{Enomoto1972}. We identify $T\cap \Sp(4, q)$ with $\Fq^\times = \langle \gamma\rangle$ if $\varepsilon(acb^{-2})=1$ and with $\langle \eta\rangle$ if $\varepsilon(acb^{-2})=-1$. 
 Arguing in a similar manner to the odd $q$ case, we obtain \begin{eqnarray}
    S_2&=& \frac{1}{q^3|T|} \sum\limits_{t\in T\setminus Z} \sum\limits_{x,y,z\in \Fq} \chi(t)^{-1}\omega( \mu(t)I_4 )^{1/2} \psi(-a x -by-cz)\theta( \mu(t)^{-1/2}t\begin{bsmallmatrix}1&&y&z\\&1&x&y\\&&1\\&&&1\end{bsmallmatrix} )  \nonumber \\
    &=& \frac{q-1}{q^3|T|} \sum\limits_{t\in (T\setminus Z)\cap \Sp(4, q)} \sum\limits_{x,y,z\in \Fq} \chi(t)^{-1} \psi(-a x -by-cz)\theta( t\begin{bsmallmatrix}1&&y&z\\&1&x&y\\&&1\\&&&1\end{bsmallmatrix} )  \nonumber \\
    &=& \begin{cases}
   \frac{1}{q(q-1)} \sum\limits_{i=1}^{q-2} \chi^{-1}(\gamma^i) \left[ \theta (C_2(i)  ) - \theta (D_2(i) )  \right]   & \text{ if } \varepsilon(acb^{-2})=1 \\
   \frac{1}{q(q+1)} \sum\limits_{i=1}^q \chi^{-1}(\eta^i) \left[ \theta (C_4(i)  ) - \theta (D_4(i) )  \right]  & \text{ if } \varepsilon(acb^{-2})=-1. 
    \end{cases}
    \end{eqnarray} The value of $S_2$ can now readily be found from the character tables in \cite{Enomoto1972}.

%
 \end{proof}

\begin{corollary} \label{corollary}
Let $(\sigma, V_\sigma)$ be an irreducible representation of $\GSp(4, q)$, and $a,b,c\in \Fq$ with $b^2-4ac\neq 0$. Let $R=TN$ be the associated Bessel subgroup and $(\chi\otimes \psi_{a,b,c},  \C_{\chi, a,b,c})$ a one-dimensional representation of $R$. 

a) Suppose $T\cong \Fq^\times \times \Fq^\times$. If $\sigma$ is generic then \begin{equation*}
1\leq \dim\Hom_R(V_\sigma, \C_{\chi, a,b,c})\leq 2
\end{equation*} and if $\sigma$ is nongeneric then \begin{equation*}
 \dim\Hom_R(V_\sigma, \C_{\chi, a,b,c})\leq 1.
\end{equation*} In addition, if $\sigma$ is nongeneric then $\dim\Hom_R(V_\sigma, \C_{\chi, a,b,c})= 1$ for at most one $\chi$, except for $\chi_7(k)$ (when $q$ is even) and $\chi_1(\lambda, \nu)$ (when $q$ is odd), in which case there are two such $\chi$. 

b) If $T\cong \mathbb{F}_{q^2}^\times$  then  $$\dim\Hom_R(V_\sigma, \C_{\chi, a,b,c})\leq 1.$$ In addition, if $\sigma$ is nongeneric then $\dim\Hom_R(V_\sigma, \C_{\chi, a,b,c})= 1$ for at most one $\chi$, except for $\chi_{9}(k)$ (when $q$ is even) and $\chi_7(\Lambda)$ (when $q$ is odd) in which case there are two such $\chi$. 
\end{corollary}

\begin{table}
 \caption{The irreducible characters of $\GSp(4,q)$ for $q$ odd, as listed in \cite{Shinoda1982}. If $(\sigma,V_\sigma)$ is the representation in the first column, then the ``rank $0$'' column shows $\dim\Hom_N(V_\sigma,\C_{0,0,0})$. The ``rank $1$'' column shows $\dim\Hom_N(V_\sigma,\C_{a,b,c})$ if $b^2-4ac=0$ and $(a,b,c)\neq(0,0,0)$. The ``rank $2$ square'' column shows $\dim\Hom_N(V_\sigma,\C_{a,b,c})$ if $b^2-4ac\in\Fq^{\times2}$. The ``rank $2$ non-square'' column shows $\dim\Hom_N(V_\sigma,\C_{a,b,c})$ if $b^2-4ac\in\xi\Fq^{\times2}$. The last two columns indicate the cuspidal and generic representations.}\vspace{-1ex}\label{GSp4qoddtable}
$$
 \begin{array}{ccccccccc}
  \toprule
   \sigma&\dim V_\sigma&\text{rank }0&\text{rank }1&\multicolumn{2}{c}{\text{rank }2}&\text{cusp}&\text{gen}\\
   &&&&\text{square}&\text{non-square}&\\
  \toprule
   X_1(\lambda,\mu, \nu)&(q+1)^2(q^2+1)&4(q+1)&q+3&q+3&q+1&&\bullet\\
  \midrule
   X_2(\Lambda, \nu)&q^4-1&2(q-1)&q-1&q+1&q-1&&\bullet\\
  \midrule
   X_3(\Lambda,\nu)&q^4-1&0&q+1&q-1&q+1&&\bullet\\
  \midrule
   X_4(\Theta)&(q^2-1)^2&0&q-1&q-1&q+1&\bullet&\bullet\\
  \midrule
   X_5(\Lambda,\omega)&(q-1)^2(q^2+1)&0&q-1&q-1&q-3&\bullet&\bullet\\
  \midrule
   \chi_1(\lambda, \nu)&(q+1)(q^2+1)&2(q+1)&1&2&0\\
  \midrule
   \chi_2(\lambda, \nu)&q(q+1)(q^2+1)&2(q+1)&q+2&q+1&q+1&&\bullet\\
  \midrule
   \chi_3(\lambda, \nu)&(q+1)(q^2+1)&q+3&2&1&1\\
  \midrule
   \chi_4(\lambda, \nu)&q(q+1)(q^2+1)&3q+1&q+1&q+2&q&&\bullet\\
  \midrule
   \chi_5(\omega, \nu)&(q-1)(q^2+1)&q-1&0&1&1\\
  \midrule
   \chi_6(\omega, \nu)&q(q-1)(q^2+1)&q-1&q-1&q&q-2&&\bullet\\
  \midrule
   \chi_7(\Lambda)&(q-1)(q^2+1)&0&1&0&2\\
  \midrule
   \chi_8(\Lambda)&q(q-1)(q^2+1)&0&q&q-1&q-1&&\bullet\\
  \midrule
   \tau_1(\nu)&q^2+1&2&1&0&0\\
  \midrule
   \tau_2(\nu)&q(q^2+1)&q+1&1&1&1\\
  \midrule
   \tau_3(\nu)&q^2(q^2+1)&2q&q&q+1&q-1&&\bullet\\
  \midrule
   \tau_4(\lambda')&q^2-1&0&1&0&0\\
  \midrule
   \tau_5(\lambda')&q^2(q^2-1)&0&q&q-1&q+1&&\bullet\\
  \midrule
   \theta_1(\nu)&\frac12q(q+1)^2&q+1&1&1&0\\
  \midrule
   \theta_2(\nu)&\frac12q(q-1)^2&0&0&0&1&\bullet&\\
  \midrule
   \theta_3(\nu)&\frac12q(q^2+1)&1&1&0&1\\
  \midrule
   \theta_4(\nu)&\frac12q(q^2+1)&q&0&1&0\\
  \midrule
   \theta_5(\nu)&q^4&q&q&q&q&&\bullet\\
  \midrule
   \theta_0(\nu)&1&1&0&0&0\\
  \bottomrule
 \end{array}
$$
\end{table}

\begin{table}
 \caption{The irreducible characters of $\GSp(4,q)$ for $q$ odd, as listed in \cite{Shinoda1982}. If $(\sigma,V_\sigma)$ is the representation in the first column, the remainder shows $\dim\Hom_R(V_\sigma,\C_{\chi, a,b,c})$ assuming $\chi|_Z = \omega_\sigma$. In the square column, in which $b^2-4ac\in \Fq^{\times 2}$, 
 we write $\chi_+= \chi|_{T_+}$ and use the canonical isomorphism $T_+\cong \Fq^\times$ to identify $\chi_+$ with a character of $\Fq^\times$. In the square column the inner product is for characters of $\Fq^\times$ while in the nonsplit column it is for $\mathbb{F}_{q^2}^\times$. By $N:\mathbb{F}_{q^2}^\times \to \Fq^\times$ we denote the norm map.  }\vspace{-1ex}\label{Bessel odd}
$$
 \begin{array}{ccc}
  \toprule
   \sigma& \text{square} &\text{nonsquare} \\
  \toprule
   X_1(\lambda,\mu, \nu)& 1+(\chi_+, \nu+\nu\mu+\nu\lambda+\nu\mu\lambda)  & 1\\
  \midrule
   X_2(\Lambda, \nu)& 1+(\chi_+, \nu+\nu\Lambda|_{\Fq^\times}) & 1-(\chi, \Lambda(\nu\circ N) + \Lambda^q (\nu\circ N))\\
  \midrule
   X_3(\Lambda,\nu)& 1&1 \\
  \midrule
   X_4(\Theta)&1 & 1\\
  \midrule
   X_5(\Lambda,\omega)&1 & 1-(\chi, \Lambda + \Lambda^q + \Lambda\tilde{\omega}+\Lambda^q\tilde{\omega}^q) \\
  \midrule
   \chi_1(\lambda, \nu)& (\chi_+, \nu+\nu\lambda) & 0 \\
  \midrule
   \chi_2(\lambda, \nu)&1+(\chi_+, \nu+\nu\lambda)  &1 \\
  \midrule
   \chi_3(\lambda, \nu)& (\chi_+, \nu\lambda) & (\chi, (\lambda\nu)\circ N) \\
  \midrule
   \chi_4(\lambda, \nu)& 1+(\chi_+, \nu+\nu\lambda+\nu\lambda^2) & 1- (\chi, (\lambda\nu)\circ N)\\
  \midrule
   \chi_5(\omega, \nu)& (\chi_+, \nu) & (\chi, \nu\circ N) \\
  \midrule
   \chi_6(\omega, \nu)& 1+(\chi_+,\nu)& 1-(\chi, \nu\circ N + \tilde{\omega}\nu\circ N + \tilde{\omega}^q\nu\circ N) \\
  \midrule
   \chi_7(\Lambda)& 0& (\chi, \Lambda+\Lambda^q) \\
  \midrule
   \chi_8(\Lambda)&1  & 1-(\chi, \Lambda+\Lambda^q) \\
  \midrule
   \tau_1(\nu) & 0 & 0\\
  \midrule
   \tau_2(\nu) & (\chi_+, \nu) & (\chi, \nu\circ N) \\
  \midrule
   \tau_3(\nu )& 1+(\chi_+, \nu+\nu\alpha_0) & 1-(\chi, \nu\circ N + (\nu\alpha_0)\circ N)\\
  \midrule
   \tau_4(\lambda')& 0 & 0\\
  \midrule
   \tau_5(\lambda')& 1 & 1\\
  \midrule
   \theta_1(\nu)& (\chi_+, \nu) & 0\\
  \midrule
   \theta_2(\nu)& 0  & (\chi, \nu\circ N) \\
  \midrule
   \theta_3(\nu)&  0& (\chi, \nu\circ N)\\
  \midrule
   \theta_4(\nu)& (\chi_+, \nu) & 0 \\
  \midrule
   \theta_5(\nu)& 1+(\chi_+, \nu) & 1-(\chi, \nu\circ N) \\
  \midrule
   \theta_0(\nu)& 0 & 0 \\
  \bottomrule
 \end{array}
$$
\end{table}

\begin{table}
 \caption{The irreducible characters of $\GSp(4,q)$ for $q$ even. Shown are the irreducible characters of $\Sp(4,q)$ as determined in \cite{Enomoto1972}. The irreducible characters of $\GSp(4,q)$ follow from $\GSp(4,q)=\Fq^\times\times\Sp(4,q)$. If $(\sigma,V_\sigma)$ is the representation in the first column, then colums 3, 4, 5, 6 show $\dim\Hom_N(V_\sigma,\C_{a,b,c})$ under the indicated conditions. The last columns indicate the cuspidal and generic representations.}\label{GSp4qeventable}
$$
 \begin{array}{ccccccccc}
  \toprule
   \sigma&\dim V_\sigma&\multicolumn{2}{c}{b=0}&\multicolumn{2}{c}{b\neq0}&\!\!\text{cusp}\!&\!\text{gen}\!\!\\
   &&\!(a,c)=(0,0)\!&\!\!(a,c)\neq(0,0)\!\!&\!\varepsilon(\frac{ac}{b^2})=1\!&\!\varepsilon(\frac{ac}{b^2})=-1\!&\\
  \toprule
   \theta_0&1&1&0&0&0\\
  \midrule
     \theta_1&\frac12q(q+1)^2&q+1&1&1&0\\
  \midrule
   \theta_2&\frac12q(q^2+1)&1&1&0&1\\
  \midrule
   \theta_3&\frac12q(q^2+1)&q&0&1&0\\
  \midrule
   \theta_4&q^4&q&q&q&q&&\bullet\\
  \midrule
   \theta_5&\frac12q(q-1)^2&0&0&0&1&\bullet&\\
  \midrule
   \chi_1(k,l)&(q+1)^2(q^2+1)&4(q+1)&q+3&q+3&q+1&&\bullet\\
  \midrule
   \chi_2(k)&q^4-1&2(q-1)&q-1&q+1&q-1&&\bullet\\
  \midrule
   \chi_3(k,l)&q^4-1&0&q+1&q-1&q+1&&\bullet\\
  \midrule
   \chi_4(k,l)&(q-1)^2(q^2+1)&0&q-1&q-1&q-3&\bullet&\bullet\\
  \midrule
   \chi_5(k)&(q^2-1)^2&0&q-1&q-1&q+1&\bullet&\bullet\\
  \midrule
   \chi_6(k)&(q+1)(q^2+1)&q+3&2&1&1\\
  \midrule
   \chi_7(k)&(q+1)(q^2+1)&2(q+1)&1&2&0\\
  \midrule
   \chi_8(k)&(q-1)(q^2+1)&q-1&0&1&1\\
  \midrule
   \chi_9(k)&(q-1)(q^2+1)&0&1&0&2\\
  \midrule
   \chi_{10}(k)&q(q+1)(q^2+1)&3q+1&q+1&q+2&q&&\bullet\\
  \midrule
   \chi_{11}(k)&q(q+1)(q^2+1)&2(q+1)&q+2&q+1&q+1&&\bullet\\
  \midrule
   \chi_{12}(k)&q(q-1)(q^2+1)&q-1&q-1&q&q-2&&\bullet\\
  \midrule
   \chi_{13}(k)&q(q-1)(q^2+1)&0&q&q-1&q-1&&\bullet\\
  \bottomrule
 \end{array}
$$
\end{table}

\begin{table}
 \caption{The irreducible characters of $\GSp(4,q)$ for $q$ even. Shown are the irreducible characters of $\Sp(4,q)$ as determined in \cite{Enomoto1972}. The irreducible characters of $\GSp(4,q)$ follow from $\GSp(4,q)=\Fq^\times\times\Sp(4,q)$. If $(\sigma,V_\sigma)$ is the representation in the first column, the next columns give $\dim\Hom_R(V_\sigma,\C_{\chi, a,b,c})$ assuming $\chi|_Z = \omega_\sigma$.  If $\varepsilon(acb^{-2})=1$ then we identify $T\cap \Sp(4, q)$ with $\langle \gamma \rangle$, define the index $j$ by $\chi(\gamma) = \zeta_{q-1}^j$, and $\delta_{y,z}$ denotes the Kronecker delta over $\Z/(q-1)\Z$. If $\varepsilon(acb^{-2})=-1$ then we identify $T\cap \Sp(4, q)$ with $\langle \eta \rangle$, define the index $j$ by $\chi(\gamma) = \zeta_{q+1}^j$, and $\delta_{y,z}$ denotes the Kronecker delta over $\Z/(q+1)\Z$.}\vspace{-1ex} \label{Bessel even}
$$
 \begin{array}{ccc}
  \toprule
   \sigma& \varepsilon(acb^{-2})=1 & \varepsilon(acb^{-2})=-1\\
  \toprule
   \theta_0  & 0 & 0 \\
  \midrule
   \theta_1& \delta_{j,0}  & 0 \\
  \midrule
   \theta_2& 0 & \delta_{j,0} \\
  \midrule
   \theta_3& \delta_{j,0} & 0 \\
  \midrule
   \theta_4& 1+\delta_{j,0} & 1-\delta_{j,0} \\
  \midrule
   \theta_5& 0 & \delta_{j,0} \\
  \midrule
   \chi_1(k,l)& 1+\delta_{j, k\pm l}+\delta_{j, -k\pm l} & 1\\
  \midrule
   \chi_2(k)& 1+\delta_{j, \pm k} & 1-\delta_{j,\pm k} \\
  \midrule
   \chi_3(k,l)& 1& 1\\
  \midrule
   \chi_4(k,l)&1 & 1-\delta_{j, k\pm l}-\delta_{j, -k\pm l}\\
  \midrule
   \chi_5(k)&1 & 1\\
  \midrule
   \chi_6(k)& \delta_{j,0} & \delta_{j,0} \\
  \midrule
   \chi_7(k)& \delta_{j, \pm k} & 0\\
  \midrule
   \chi_8(k)&\delta_{j,0} &  \delta_{j,0}\\
  \midrule
   \chi_9(k)& 0 & \delta_{j, \pm k } \\
  \midrule
   \chi_{10}(k)& 1+\delta_{j,0}+\delta_{j, \pm k} & 1-\delta_{j,0} \\
  \midrule
   \chi_{11}(k)&1+\delta_{j,\pm k} & 1\\
  \midrule
   \chi_{12}(k)& 1+\delta_{j,0}& 1-\delta_{j,0}-\delta_{j, \pm 2k} \\
  \midrule
   \chi_{13}(k)& 1 & 1- \delta_{j, \pm k} \\
  \bottomrule
 \end{array}
$$
\end{table}

\end{document}